\documentclass[a4paper,11pt]{article}
\usepackage{graphicx} 
\usepackage{amsmath}
\usepackage{amssymb}
\usepackage{amsthm}
\usepackage{geometry}
\usepackage{hyperref}
\usepackage{cite}
\usepackage{float}
\usepackage{relsize}
\usepackage{subcaption}
\usepackage{comment}
\usepackage{indentfirst}
\usepackage{xcolor}

\usepackage{vruler}

\newtheorem{theorem}{Theorem}
\newtheorem{defn}[theorem]{Definition}
\newtheorem{prop}[theorem]{Proposition}
\newtheorem{lemma}[theorem]{Lemma}
\newtheorem{cor}[theorem]{Corollary}
\newtheorem{exa}{Example}[section]
\newtheorem{remark}[theorem]{Remark}
\newtheorem{conjecture}[theorem]{Conjecture}
\newtheorem{claim}{Claim}[theorem]
\newtheorem{case}{Case}[theorem]
\newtheorem{subcase}{Subcase}[case]

\newcommand{\Define}[1]{\textbf{#1}}

\title{On the maximum and negative frustration indices of graphs}
\author{Maximilien Gadouleau\footnote{Department of Computer Science, Durham University, Durham, UK. \texttt{m.r.gadouleau@durham.ac.uk}} \and Huiying Zeng\footnote{Department of Computer Science, Durham University, Durham, UK. \texttt{huiying.zeng@durham.ac.uk}}}
\date{\today}

\begin{document}

\maketitle

\begin{abstract}
    A signed graph is a  graph with signatures ($+1$ or $-1$) on its edges. A cycle is called positive if the product of its edge signatures is positive, and a signed graph is called balanced if each cycle in it is positive. The frustration index is the minimum number of edges whose deletion makes the signed graph balanced, which is considered to be a measurement of the imbalance of the signed graph. In this paper, we compare the frustration index of the all-negative signature with the maximum frustration index of all possible signatures on the unsigned graph. We classify some families of graphs into three scenarios: the all-negative signature does not maximise the frustration index, the all-negative signature maximises the frustration index non-uniquely, and the all-negative signature maximises the frustration index uniquely. For all three scenarios, we can exhibit chordal and non-chordal graphs alike. The classes we consider include apex trees, fan graphs, wheel graphs, and complete split graphs. Moreover, for the families of fan graphs and wheel graphs, we fully characterise and count the signatures maximising the frustration index. Throughout our study, we exhibit different classes of signed graphs for which the frustration index equals the number of edge-disjoint negative triangles. Moreover, as part of our study, we are able to refute three conjectures of Zaslavsky on the frustration index.
\end{abstract}

\section{Introduction} \label{section:introduction}

\subsection{Background} \label{subsection:background}

\paragraph{Signed graphs and their frustration index}

A signed graph $\Sigma = (G, \sigma)$ is a graph $G$ with a sign $\sigma(e) \in \{+1, -1\}$ on all its edges.
The sign of a cycle is the product of the signs of its edges.
Therefore, a cycle is positive (negative, respectively) if and only if it contains an even (odd, respectively) number of negative edges.
A signed graph is called balanced if it has no negative cycles.
Balanced signed graphs are pivotal to the theory of signed graphs, as they can be viewed as the natural analogue to bipartite graphs. 
Indeed, if $\Sigma$ is balanced, then its vertex set can be partitioned into $V = S \cup T$ with only positive edges inside $S$ and inside $T$, and only negative edges between $S$ and $T$.

The frustration index $l( \Sigma )$ is the minimum number of edges to remove from $\Sigma$ to obtain a balanced signed graph.
It is arguably the most important measure of imbalance for signed graphs, and as such has received a lot of attention \cite{Iacono, Katai, Sole, Bowlin, Martin, Sehrawat, Aref, Diaz, Chen, Shahul}.

Let us mention one important lower bound on the frustration index.
Let $p^-( \Sigma )$ ($p^-_\triangle( \Sigma )$, respectively) denote the maximum number of edge-disjoint negative cycles (triangles, respectively) in $\Sigma$. It is easy to see that
\begin{equation} \label{equation:p_v_l}
    p^-_\triangle( \Sigma ) \le p^-( \Sigma ) \le l( \Sigma ).
\end{equation}

\paragraph{Maximum and negative frustration indices of unsigned graphs}

Let $G$ be an unsigned graph. 
The maximum frustration index of $G$, denoted by $l_{\max}( G )$, is simply the maximum frustration index of a signature of $G$.

A naive way to obtain high frustration index consists in setting all edges to be negative, thus obtaining the signed graph $-G$.
The negative frustration index of $G$ is then defined as $l(-G)$.
It is actually equal to $|E(G)| - \text{MaxCut}(G)$, thus showing that computing the frustration index of signed graphs is NP-hard.

It is easy to find examples of graphs where the all-negative signature does not maximise the frustration index: for instance, if $G$ is a bipartite graph with a cycle, then $-G$ is balanced, while negating only one edge yields an unbalanced graph. 
Conversely, Petersdorf \cite{Petersdorf} showed that for complete graphs, the all-negative signature is the unique (up to switching-equivalence) signature that maximises the frustration index. 

Inspired by these results, Zaslavsky conjectured that if $G$ is chordal, then the negative signature
maximises the frustration index, i.e. $l_{\max}( G ) = l( -G )$. This conjecture, and two other ones by Zaslavsky (Conjectures \ref{conjecture:chordal}, \ref{conjecture:multipartite}, and \ref{conjecture:disjoint_cycles}), will be refuted in this paper.

\subsection{Contributions} \label{subsection:contributions}







\paragraph{Main contributions}
In this paper, we compare the frustration index of the negative signature (negative frustration index) with the maximum frustration index. 
In this paper, we show that for chordal and non-chordal graphs alike, all three scenarios can occur.
\begin{enumerate}
    \item The negative signature does not maximise the frustration index. \\
    Chordal example: any graph of the form $K_1 \vee S_n = S_2^{n-1}$, which is both an apex tree and a complete split graph  \textit{(new)}. \\
    Non-chordal example: any bipartite graph that is not a tree.

    \item The negative signature does maximise the frustration index, but not uniquely, i.e. there is another switching class that also maximises the frustration index. \\
    Chordal example: any fan graph $F_n$ \textit{(new)}. \\
    Non-chordal example: any wheel $W_{2k}$ with an even number of spokes \textit{(new)}.

    \item The negative signature maximises the frustration index uniquely, i.e. its switching class is the only one that maximises the frustration index. \\
    Chordal example: any complete graph $K_n$, any complete split graphs $S_n^2$ for $n \ge 3$ or $S_n^3$ for $n \ge 5$ \textit{(new)}. \\
    Non-chordal example: any odd cycle $C_{2k+1}$, any wheel $W_{2k+1}$ with an odd number of spokes \textit{(new)}.
\end{enumerate}
In particular, we exhibit the first examples of chordal graphs in the first scenario and the first examples of graphs in the second scenario, chordal and non-chordal alike.



\paragraph{Detailed contributions}
In more detail, in Section \ref{section:apex_trees} we prove that for any signed apex tree $A_n = A_n(h, T)$ (i.e. removing the vertex $h$ leaves out the tree $T$), $l_{\max}( A_n ) = \lfloor \frac{n}{2} \rfloor$ is given by the maximum number of edge- and endpoint-disjoint paths in $T$, while $l( - A_n )$ is given by the maximum number of edge- and endpoint-disjoint paths of odd length in $T$. In particular, for the star $S_n$, $l( -S_n ) = 1$, thus $K_1 \vee S_n$ is a family of chordal graphs whose negative frustration index is bounded (at $1$), while the maximum frustration index is unbounded. Moreover, we show that there is no apex tree (unless $T \in \{K_1, K_2\}$) such that the all-negative signature maximises the frustration index uniquely. We then study fan graphs, i.e. $F_n = K_1 \vee P_n$ where $P_n$ denotes the path on $n$ vertices. We are able to classify and count the switching classes that maximise the frustration index of the fan graph.

We move on to study wheel graphs in Section \ref{section:wheel_graphs}, i.e. $W_n = K_1 \vee C_n$ where $C_n$ denotes the cycle on $n$ vertices. We first show that $l( -W_n ) = l_{\max}( W_n ) = \lceil \frac{n}{2} \rceil$ for all $n$. We then show that the all-negative signature uniquely maximises the frustration index when $n$ is odd. On the other hand, for $n$ even this is not the case, and we characterise and count the number of switching classes that maximise the frustration index.

We finally study complete split graphs in Section \ref{section:complete_split_graphs}. We determine the negative frustration index of $S_n^m = K_n \vee \overline{K_m}$. We then prove that the all-negative signature maximises the frustration index uniquely for $S_n^2$ with $n \ge 3$ and for $S_n^3$ with $n \ge 5$.

\paragraph{Additional results}
An important problem is to exhibit signed graphs $\Sigma$ for which the frustration index $l( \Sigma )$ actually equals the number of edge-disjoint negative cycles $p^-( \Sigma )$.
As part of our study, we exhibit classes of signed graphs for which the bound in \eqref{equation:p_v_l} is tight: either $l( \Sigma ) = p^-( \Sigma )$ or even $l( \Sigma ) = p^-_\triangle( \Sigma )$. In more detail, we prove that: for any signed apex tree, $l( A_n, \sigma ) = p^-( A_n, \sigma )$; for any signed fan graph, $l( F_n, \sigma ) = p^-_\triangle( F_n, \sigma )$; and for any signed wheel graph, $l( W_n, \sigma ) = p^-_\triangle( W_n, \sigma )$ unless $(W_n, \sigma) = -W_{2k+1}$.

And finally, as mentioned above, in this paper we refute three conjectures of Zaslavsky on the frustration index.

\section{Preliminaries} \label{section:preliminaries}

\subsection{Signed graphs} \label{subsection:signed_graphs}

In this section, we review some basic concepts of signed graph theory. The reader is referred to \cite{Zaslavsky} for more details. A \textit{signed graph} $\Sigma$ is a pair $(G, \sigma)$ where 
$G = ( V, E )$ is a graph, called the \textit{underlying graph}, which we denote by $|\Sigma|$ and $\sigma: E \rightarrow \{+1, -1\}$ is the \textit{signature}. When the underlying unsigned graph is clear, we shall use the terms ``signed graph'' and ``signature'' interchangeably. The two simplest signed graphs on $G$ are the positive graph $+G$ (with $\sigma = +1$) and the negative graph $-G$ (with $\sigma = -1$).

For a vertex subset $X$ in $\Sigma$, we denote
\begin{align*}
    \partial X &= \{ e \in E: |e \cap X| = 1 \}, \\
    \partial^+ X &= \{ e \in \partial X, \sigma( e ) = +1 \}, \\
    \partial^- X &= \{ e \in \partial X, \sigma( e ) = -1 \}.
\end{align*}
We further denote $d( X ) = | \partial X |$, $d^+( X ) = | \partial^+ X |$, and $d^-( X ) = | \partial^- X |$. In particular, the \textit{positive degree} of a vertex $v \in V$ is the number of positive edges incident to $v$, and the \textit{negative degree} of a vertex $v\in V(\Sigma)$ is the number of negative edges incident to $v$, which we denote by $d^+(v)$ and $d^-(v)$ respectively. The \textit{degree} of $v$ is $d(v)=d^+(v)+d^-(v)$, which is the number of edges incident to $v$.




The sign of a cycle $C = e_1e_2\cdots e_l$ is the product of its edge signs: $\sigma(C):=\sigma(e_1)\sigma(e_2)\cdots\sigma(e_l).$ \textit{Switching} a signed graph $\Sigma$ means negating the signs of all edges between a vertex subset $X$ and its complement. The switched graph is written $\Sigma^X=(|\Sigma|, \sigma^X)$. We have
\[
    \sigma^X( e ) = 
    \begin{cases}
        - \sigma( e ) &\text{if } e \in \partial X \\
        \sigma( e ) &\text{otherwise}.
    \end{cases}
\]
   
Switching does not change the sign of any cycle. If $\Sigma$ can be switched to become $\Sigma'$, we say $\Sigma$ and $\Sigma'$ are \textit{switching equivalent}. A class of signed graphs under switching equivalence is called a \textit{switching class}. 

For a spanning tree $T$ in a graph $G$, an edge $e$ in $G$ but not in $T$ is called a \textit{chord}. A cycle formed by adding a chord to the spanning tree is called a \textit{fundamental cycle}. Given any spanning tree, the corresponding fundamental cycle set has size $(|E(G)| - |V(G)| +1)$. Different spanning trees induce different fundamental cycle sets. Given any spanning tree and the corresponding fundamental cycle set, two signed graphs with the same underlying graph are switching equivalent if and only if they share the same set of negative fundamental cycles.

A signed graph $\Sigma$ is called \textit{balanced} if every cycle in it is positive. Equivalently, $\Sigma = ( G, \sigma )$ is balanced if $\Sigma$ is in the switching class of $+G$, i.e. there exists $X \subseteq V$ such that $\sigma( e ) = -1$ if and only if $e \in \partial X$.  Similarly, a signed graph is \textit{antibalanced} if it belongs to the switching class of $-G$, i.e. there exists $X \subseteq V$ such that $\sigma( e ) = +1$ if and only if $e \in \partial X$. This is equivalent to the definition that a signed graph $\Sigma$ is antibalanced if $-\Sigma$ is balanced. We note that an antibalanced graph does not contain any positive triangles.

\subsection{Frustration index} \label{subsection:frustration_index}

The \textit{frustration index} of $\Sigma$ is the minimum number of edges whose deletion makes $\Sigma$ balanced, which we denote by $l(\Sigma)$. Equivalently, the frustration index is equal to the minimum number of edges whose negation makes $\Sigma$ balanced. A signed graph is balanced if and only if its frustration index equals to $0$. 

Let us review two simple bounds on the frustration index. In a signed graph $\Sigma$, we must at least negate one edge in each edge-disjoint negative cycles in order to eliminate those negative cycles. In addition, negating all negative edges in $\Sigma$ will make it balanced. Let $p^-( \Sigma )$ be the number of edge-disjoint negative cycles in $\Sigma$. We obtain the following bounds for the frustration index:
\begin{equation}\label{eq:bounds}
    p^-(\Sigma) \le l(\Sigma) \le | E^-(\Sigma) |.
\end{equation}
Replacing a positive edge with a path of two negative edges in a signed graph does not change the maximum number of edge-disjoint negative cycles and the frustration index, which are direct results from the sign-preserving cycle bijection in \cite{Zaslavsky2018}.

Switching does not change the frustration index, but it can change the number of negative edges. In fact, the frustration index of $\Sigma$ is the smallest number of negative edges in a switching of $\Sigma$\cite{Zaslavsky2018}: 
\[
    l( \Sigma ) = \min \left\{ | E^-( \Sigma^S ) | : S \subseteq V \right\}.
\]

\begin{defn}
    A signed graph is called minimum if the number of negative edges in it equals its frustration index.
\end{defn}
In each switching class, there is at least one minimum signed graph. The following is an important theorem to find maximum frustration.
\begin{theorem}\label{thm:properly switched}\cite{Iacono}
    A signed graph $\Sigma$ is minimum if and only if $d^-(S) \le d^+(S)$ for each $S \subseteq V$.
\end{theorem}

\begin{theorem}\label{thm:spanningtree}
    \cite{Zaslavsky1982}~Given a signed graph $\Sigma$ and a spanning tree $T$ of $\Sigma$, there exists a switched graph $\Sigma'$ such that $\Sigma'$ has any desired signs on $T$.
\end{theorem}

For a graph $G$, the frustration index of the all-negative signed graph $-G$ is given by 
\[
    l(-G) = |E(G)| - \max_{X \subseteq V(G)}|E(X, \overline{X})|,
\]
the complement of the maximum cut size in $G$ \cite{Zaslavsky2018}. We call it the \textit{negative frustration index} of $G$. In particular, this shows that computing the frustration index of signed graphs is NP-hard.
    
For a graph $G$, the maximum frustration index over all signatures is denoted by
\[    
    l_{\max} (G) := \max_{\sigma: E(G)\rightarrow\{+1, -1\}}l(G, \sigma).
\]
We call it the \textit{maximum frustration index} of $G$.

The maximum frustration index and negative frustration index are based on the cycle structure of the graph; moreover, because subdividing an edge does not change the maximum frustration index. As such, henceforth we only consider graphs that are $2$-connected and where each vertex of degree $2$ is contained in a triangle.

\subsection{Negative frustration v maximum frustration} \label{subsection:negative_v_maximum}

Our main goal in this paper is to compare the negative frustration index $l( -G )$ to the maximum frustration index $l_{\max}( G )$. It is clear that not all graphs $G$ satisfy $l( -G ) = l_{\max}( G )$; for instance, if $G$ is a non-forest bipartite graph, then $l_{\max}( G ) > 0 = l( -G )$. On the other hand, Petersdorf proved that this was the case for the complete graphs.

\begin{theorem}\label{thm:complete}
    \cite{Petersdorf}For the complete graph $K_n$ with $n$ vertices,
    \[
        l_{\max}(K_n) = l(-K_n) = \left\lfloor \frac{(n-1)^2}{4} \right\rfloor.
    \]
    Additionally, the signatures whose frustration index achieves the maximum are precisely those in the switching class of $-K_n$.
\end{theorem}

Based on Petersdorf's result, Zaslavsky raised two conjectures regarding maximum frustration index of some graph classes. The first one on chordal graphs appears on his webpage, while the second one appears in \cite{Zaslavsky}.


\begin{conjecture}\label{conjecture:chordal}
    Chordal graphs attain maximum frustration index when signed all negative.
\end{conjecture}
\begin{conjecture}\label{conjecture:multipartite}
    \cite{Zaslavsky} Complete multipartite graphs $K_{n_1, n_2, \cdots, n_k}$ where $k\ge 3$ attain maximum frustration index when signed all negative.
\end{conjecture}

Unfortunately, the chordal graph $K_{3,1,1}$ is a counterexample to both conjectures, since
\[
    l( -K_{3,1,1} ) = 1 < 2 = l_{\max}( K_{3,1,1} ).
\]

Denote the maximum number of edge-disjoint cycles (triangles, resp.) in a graph $G$ by $p( G )$ ($p_\triangle( G )$, resp.). This can be naturally generalised to signed graphs, as $p( \Sigma ) = p( |\Sigma| )$ and $p_\triangle( \Sigma ) = p_\triangle( |\Sigma| )$. We then have $p^-( \Sigma ) \le l( \Sigma )$ as above, while the inequality $p^-( \Sigma ) \le p( \Sigma )$ follows from the definition. Zaslavsky conjectured that we always have equality in either of those two inequalities.
\begin{conjecture}\label{conjecture:disjoint_cycles} \cite{Zaslavsky2018}
    \[
        p^-(\Sigma) = \min (l(\Sigma), p(\Sigma)).
    \]
\end{conjecture}
This conjecture is also false, as we shall prove in Section \ref{section:wheel_graphs}. In fact, our counterexample satisfies $p^-(\Sigma) < \min (l(\Sigma), p_\triangle(\Sigma))$.




\subsection{Notation} \label{subsection:notation}

The scope of this paper is families of graphs (not necessarily chordal) that contain a dominating vertex.  For two vertex-disjoint graphs $G_1$ and $G_2$, the join $G_1\vee G_2$ is a new graph with the vertex set $V(G_1)\cup V(G_2)$, and the edge set of $G_1\vee G_2$ is exactly $E(G_1)\cup E(G_2)$ as well as the edges connecting all vertices in $G_1$ and all vertices in $G_2$. In particular,
\begin{itemize}
    \item $F_n = K_1 \vee P_n$ is the \Define{fan graph} on $n+1$ vertices, where $P_n$ is the path on $n$ vertices; 

    \item $W_n = K_1 \vee C_n$ is the \Define{wheel graph} on $n+1$ vertices, where $C_n$ is the cycle on $n$ vertices; 

    \item $S_n^m = K_n\vee\overline{K_m} = K_1 \vee S_{n-1}^m$ is a \Define{complete split graph}.
\end{itemize}

\section{Apex trees} \label{section:apex_trees}

Let $A_n(h, T)$ be an apex tree such that removing vertex $h$ and the edges incident to it makes a tree $T$, and $d(h) = n$. Based on our assumption of only considering $2$-connected graphs, we assume that all leaves of the tree $T$ are in the neighbourhood of $h$. Sometimes we use $A_n$ instead of $A_n(h, T)$ for convenience.

\subsection{Maximum frustration index} \label{subsection:lmax_apex_tree}

We first determine the frustration index of any signed apex tree $(A_n, \sigma)$.

\begin{theorem} \label{thm:apex_characterization}
For any tree $T$ and the apex tree $A_n = A_n(h, T)$,  $l(A_n,\sigma) = p^-(A_n,\sigma)$. 
\end{theorem}

\begin{proof}
We know that $l(A_n,\sigma) \ge p^-(A_n,\sigma)$. We now prove the reverse inequality. According to Theorem~\ref{thm:spanningtree}, we assume that $(A_n,\sigma)$ is switched so that all its negative edges are in $T$ without loss of generality.

By replacing every positive edge in the tree $T$ with $-P_3$, we get a new signed apex tree $(A_n',\sigma')$ where removing the vertex $h$ leaves a tree $T'$ subdivided from $T$, and $l(A_n, \sigma) = l(A_n',\sigma')$. In $(A_n', \sigma')$, a cycle is negative if and only if it is odd.

From $T' = (V, E)$ with bipartition $V = L \cup R$, construct the directed graph $\hat{T} = ( \hat{V}, \hat{A})$ as follows. First, $\hat{V} = V \cup \{ s, t \} \cup \{ a_e, b_e : e \in E \}$. Second, 
\[
    \hat{A} = \{ (s, l) : l \in L \cap N(h) \} \cup \{ (r, t) : r \in R \cap N(h)\} \cup \{ (l, a_e), (r, a_e), (a_e, b_e), (b_e, l), (b_e, r) : e = lr, l \in L, r \in R  \}.
\]

\begin{claim} \label{claim:cycle_path_bijection}
Odd cycles in $(A_n',\sigma')$ are in bijection with the $s$-$t$ paths in $\hat{T}$.
\end{claim}

\begin{proof}   
Any $s$-$t$ path in $\hat{T}$ is of the form $P = s l_1 a_{l_1 r_1} b_{l_1 r_1} r_1 \dots a_{l_k r_k} b_{l_k r_k} r_k t$ for some $l_1, \dots, l_k \in L$ and $r_1, \dots, r_k \in R$ such that $l_1 r_1 l_2 r_2 \dots l_k r_k$ is a path in $T'$. Let $\phi( P ) = h l_1 r_1 l_2 r_2 \dots l_k r_k h$; then $P$ is an odd cycle in $(A_n',\sigma')$, and any odd cycle in $(A_n',\sigma')$ is of the form $\phi( P )$ for some $s$-$t$ path $P$ in $\hat{T}$. It is finally clear that $\phi$ is injective.
\end{proof}

\begin{claim} \label{claim:cycle_path_disjoint}
If two $s$-$t$ paths in $\hat{T}$, $P$ and $P'$, are arc-disjoint, then their corresponding odd cycles in $(A_n',\sigma')$, $\phi( P )$ and $\phi( P' )$, are edge-disjoint.
\end{claim}

\begin{proof}    
Let $P = s l_1 a_{l_1 r_1} b_{l_1 r_1} r_1 \dots a_{l_k r_k} b_{l_k r_k} r_k t$ and $P' = s l'_1 a_{l'_1 r'_1} b_{l'_1 r'_1} r'_1 \dots a_{l'_j r'_j} b_{l'_j r'_j} r'_j t$. Since they are arc-disjoint, we obtain that $l_xr_x \ne l'_yr'_y$ for all $1 \le x \le k$ and $1 \le y \le j$. Thus, $\phi( P ) = h l_1 r_1 l_2 r_2 \dots l_k r_k h$ and $\phi( P' ) = h l'_1 r'_1 l'_2 r'_2 \dots l'_j r'_j h$ are edge-disjoint. 
\end{proof}

Say that an arc in $\hat{A}$ is fundamental if it belongs to 
\[
    \{ (s, l) : l \in L \cap N(h) \} \cup \{ (r, t) : r \in R \cap N(h)\} \cup \{ (a_e, b_e) : e = lr, l \in L, r \in R  \}. 
\]

\begin{claim} \label{claim:cycle_path_cover}
There exists a minimum set of arcs covering all $s$-$t$ paths in $\hat{T}$ containing only fundamental arcs. 
\end{claim}

\begin{proof}
If a set of arcs covering all $s$-$t$ paths contains an arc in $\{ (l, a_e), (r, a_e), (b_e, l), (b_e, r)\}$ for $e = lr$, then replacing that arc with $(a_e, b_e)$ yields another set of arcs covering all $s$-$t$ paths.
\end{proof}

Let $\alpha$ denote the maximum number of arc-disjoint $s$-$t$ paths in $\hat{T}$. We then have $\alpha \le p^-( A_n, \sigma )$ from Claims \ref{claim:cycle_path_bijection} and \ref{claim:cycle_path_disjoint}. Let $\beta$ denote the minimum size of a set of arcs that cover all $s$-$t$ paths. Then $l(A_n', \sigma') \le \beta$ follows from Claims \ref{claim:cycle_path_bijection} and \ref{claim:cycle_path_cover}: removing the edges corresponding a minimum set of fundamental arcs corresponds to removing edges in $(A_n', \sigma')$ that leaves the graph without any negative cycles. By Menger's theorem, $\alpha = \beta$ and hence
\[
    l( A_n', \sigma') \le \beta = \alpha \le p^-( A_n', \sigma').
\]
 Therefore,
 \[
    l(A_n, \sigma) = l( A_n', \sigma') = p^-( A_n', \sigma' ) = p^-( A_n, \sigma ).
\]
\end{proof}

For the all-negative signature, a cycle is negative if and only if it has an odd number of edges. As such, $p^-( - A_n(h,T) )$ is related to odd paths (i.e. with an odd number of edges) in $T$.

\begin{cor} \label{corollary:-An}
For an apex tree $A_n = A_n( h, T )$, the negative frustration index is given by the maximum cardinality of a collection of edge-disjoint and endpoint-disjoint odd paths in $T$.
\end{cor}

A matching in $T$ is a collection of edge-disjoint and endpoint-disjoint odd paths in $T$. As such, Corollary \ref{corollary:-An} yields $l( -A_n( h, T ) ) \ge \mu( T )$, where $\mu$ denotes the matching number. However, this lower bound is not tight in general, as we will see in Example \ref{example:double-star} below.

\medskip

We now move on to determining the maximum frustration index of apex trees.

\begin{theorem}\label{thm:apex_max_frustration}
    For an apex tree $A_n = A_n(h, T)$, the maximum frustration index is
    \[
        l_{\max}(A_n) = \left\lfloor \frac{n}{2} \right\rfloor. 
    \]
\end{theorem}

The bound is based on two lemmas. The following lemma is a generalisation of Lemma 1 in \cite{Wu}.

\begin{lemma}\label{lemma:apex_count_paths}
    For a tree $T$ and a vertex subset $S \subseteq V$ with $n$ vertices which contains all leaves of $T$, there are $\lfloor \frac{n}{2} \rfloor$ edge-disjoint and endpoint-disjoint paths whose endpoints are in $S$. 
\end{lemma}
\begin{proof}
    
    For a maximal path $(v_0, v_1, \cdots, v_k)$ such that $d(v_i) = 2$ and $v_i \notin S$ for $i = 1, 2, \cdots, k-1$, we replace this path with a single edge $(v_0, v_k)$. 
    We can do this constantly until it is never possible and call this new tree $T'$. It is not hard to see that $T'$ has the same number of edge-disjoint and endpoint-disjoint paths whose endpoints are in $S$. Additionally, in $T'$, every vertex with degree $2$ is in $S$. 
    
    We prove this theorem by induction. It is trivial when $|S| \le 3$. Assume that this theorem holds for $|S| \le n-1$. There are two cases for a leaf $v \in S$ in the tree $T'$.

    \begin{case}\label{case:tree1}
            There is a path $(u, w, v)$, where $d(w)=2$ and $d(u)\ge 2$.
        \end{case}
        
        In this case, $T'-\{w, v\}$ is a tree on $(n-2)$ vertices, where we can find at least $\lfloor\frac{n-2}{2}\rfloor$ edge-disjoint and endpoint-disjoint paths whose endpoints are in $S$. Additionally, path $(w,v)$ is edge-disjoint and endpoint-disjoint with all those $\lfloor\frac{n-2}{2}\rfloor$ paths.
        
        \begin{case}\label{case:tree2}
            Case~\ref{case:tree1} does not hold, but there is a path $(u, w,v)$, where $d(w)\ge 3$ and $d(u)=1$. 
        \end{case}
        
        In this case, $T'-\{u, v\}$ is a tree on $(n-2)$ vertices, where we can find at least $\lfloor\frac{n-2}{2}\rfloor$ edge-disjoint and endpoint-disjoint paths. Additionally, path $(u, w, v)$ is edge-disjoint and endpoint-disjoint with all those $\lfloor\frac{n-2}{2}\rfloor$ paths.
        
        In both cases above, we can find at least $\lfloor\frac{n}{2}\rfloor$ edge-disjoint and endpoint-disjoint paths whose endpoints are in $S$. Additionally, since $|S| = n$, we can find at most $\lfloor\frac{n}{2}\rfloor$ edge-disjoint and endpoint-disjoint paths whose endpoints are in $S$. Therefore, there are exactly $\lfloor \frac{n}{2} \rfloor$ edge-disjoint and endpoint-disjoint paths whose endpoints are in $S$.
\end{proof}

\begin{lemma} \label{lemma:lmax_p}
Let $G$ be an unsigned graph such that $l( G, \sigma ) = p^-( G, \sigma )$ for all signatures $\sigma$. Then $l_{\max}( G ) = p( G )$.
\end{lemma}

\begin{proof}
We only need to prove that $p( G ) = \max_\sigma p^-( G, \sigma )$. 
Firstly, we clearly have $p^-( G, \sigma ) \le p(G)$ for all $(G, \sigma)$. 
Conversely, there exists a signed graph $(G, \sigma)$ such that $p^-( G, \sigma ) = p( G )$: simply negate one edge from each edge-disjoint cycle. 
\end{proof}

\begin{proof}[Proof of Theorem \ref{thm:apex_max_frustration}]
For an apex tree $A_n$, every cycle must contain the vertex $h$ and a path in the tree $T$. Therefore, by Lemma \ref{lemma:apex_count_paths}, the maximum number of edge-disjoint cycles in an apex tree $A_n$ is
\[
    p( A_n ) = \left\lfloor \frac{n}{2} \right\rfloor.
\]
Lemma \ref{lemma:lmax_p} then yields $l_{\max}(A_n) = p(A_n) = \left\lfloor \frac{n}{2} \right\rfloor$.
\end{proof}

When $h$ is a universal vertex, $A_n(h, T) = K_1 \vee T$, which is chordal. By choosing $T$ to be a star, we obtain a family of chordal graphs where the negative frustration index is equal to $1$, while the maximum frustration index is unbounded. This can be viewed as a generalisation of the counterexample of Conjecture \ref{conjecture:chordal}, since $K_{3,1,1} = K_1 \vee S_4$.

\begin{cor} \label{corollary:star}
For the star $S_n$ on $n$ vertices, we have $l(-K_1 \vee S_n) = 1$ while $l_{\max}( K_1 \vee S_n ) = \lfloor \frac{n}{2} \rfloor$.
\end{cor}

We further illustrate our results by taking the example of the double star. 

\begin{exa} \label{example:double-star}
Let $n \ge 2$. The double star $S_{n,n}$ is the following tree on $2n+2$ vertices: $S_{n,n} = (V, E)$ with $V = \{a, b \} \cup \{c_1, \dots, c_n\} \cup \{ d_1, \dots, d_n \}$ and $E = \{ ac_1, \dots, ac_n \} \cup \{ab\} \cup \{ bd_1, \dots, bd_n \}$. We see that $\mu( S_{n,n} ) = 2$ while $\{(c_1, a, b, d_1), (c_2,a), (d_2,b) \}$ is a family of three edge- and enpoint-disjoint odd paths in $S_{n,n}$. Let $G = K_1 \vee S_{n,n}$. We obtain
\begin{alignat*}{3}
    p_\triangle( G ) &= p^-_\triangle( -G ) &&= 2, \\
    l( -G ) &= p^-( -G ) &&=3, \\
    l_{\max}( G ) &= p( G ) &&= n+1.
\end{alignat*}
\end{exa}

\medskip

We now study the signed graphs that maximise the frustration index in more detail. 
We begin with an important bound on the frustration index, which will be useful to us here and in Section \ref{section:complete_split_graphs} as well.

For a signed graph $\Sigma = (G, \sigma)$, and a vertex subset $S$, $\Sigma[S]$ denotes the induced signed subgraph by $S$, and $G[S]$ denotes the induced subgraph by $S$.

\begin{theorem}\label{thm:disjoint_vertex_subset}
    For a signed graph $\Sigma$, and the partition of its vertex set $V(\Sigma) = S \cup T$, we have
    \[
        l(\Sigma) \le l(\Sigma[S]) + l(\Sigma[T]) + \left\lfloor \frac{d(S)}{2} \right \rfloor.
    \]
\end{theorem}
\begin{proof}
    Since $S$ and $T$ are disjoint, $\Sigma$ can be switched so that $\Sigma[S]$ and $\Sigma[T]$ are minimum. If $d^-(S) >  \left\lfloor \frac{d(S)}{2} \right \rfloor$, switching $S$ leads to $d^-(S) <  \left\lfloor \frac{d(S)}{2} \right \rfloor$. Therefore, $\Sigma$ can always be switched so that
    \[
        | E^-(\Sigma) | \le l(\Sigma[S]) + l(\Sigma[T]) + \left\lfloor \frac{d(S)}{2} \right \rfloor.
    \]
\end{proof}

\begin{cor}\label{cor:disjoint_vertex_subset}
    For a graph $G$, and the partition of its vertex set $V(G) = S \cup T$, we have
    \[
        l_{\max}(G) \le l_{\max}(G[S]) + l_{\max}(G[T]) + \left\lfloor \frac{d(S)}{2} \right \rfloor.
    \]
\end{cor}

The following is a direct corollary from Theorem~\ref{thm:disjoint_vertex_subset}.
\begin{prop}\label{prop:-h-}
    For the apex tree $A_n = A_n(h, T)$, any switching class that maximises the frustration index has a minimum signature such that all negative edges are incident to vertex $h$, and $d^-(h) = \left\lfloor \frac{n}{2} \right\rfloor$. 
\end{prop}


There are two trivial cases of apex trees: $n = 1$, in which chase $A_n = K_2$ and $n=2$, in which case $A_n = K_3$. Even though the negative signature can maximise the frustration index of an apex tree $A_n(h,T)$, remarkably it cannot do so uniquely (unless we are in one the two trivial cases above).

\begin{theorem} \label{theorem:not_unique-max_apex_tree}
For any apex tree $A_n$ where $n \ge 3$,  the all-negative signature does not maximise the frustration index uniquely.
\end{theorem}

\begin{proof}
    The proof is trivial if $l( -A_n ) < \lfloor \frac{n}{2} \rfloor$. Suppose $l( -A_n ) = \lfloor \frac{n}{2} \rfloor$. According to Theorem~\ref{thm:apex_characterization}, there is a family $\mathcal{C}$ of exactly $\lfloor \frac{n}{2} \rfloor$ edge-disjoint negative cycles in $-A_n$, or in other words a family $\mathcal{P}$ of exactly $\lfloor \frac{n}{2} \rfloor$ edge- and endpoint-disjoint odd paths in $T$, whose endpoints are in $N(h)$. $\mathcal{C}$ and $\mathcal{P}$ might not be unique. There are three cases.
    \begin{case}\label{case:2_nonleaves}
        There are at least two non-leaves of $T$ in $N(h)$.
    \end{case}
    Since at least $n-1$ vertices in $N(h)$ are endpoints of some odd path in $\mathcal{P}$ and there are at least two non-leaves in $N(h)$, one of them must be an endpoint. Let $a$ be this non-leaf endpoint, $b$ its neighbour on the odd path, and $c$ another neighbour of $a$ in $T$. The edge $ha$ must be in at least two cycles. Let the odd cycle containing $ab$ be $C_1$, and a cycle containing edge $ha$ but not edge $ab$ be $C_2$.

    Consider the signed graph $\Sigma = ( A_n, \sigma )$, where all edges are negative apart from $ha$ and $ab$. Then $\mathcal{C}$ is a family of exactly $\lfloor \frac{n}{2} \rfloor$ edge-disjoint negative cycles in $\Sigma$ (as the sign of the cycle $C_1$ remains negative), and hence $l( \Sigma ) = \lfloor \frac{n}{2} \rfloor$. Moreover, since the cycle $C_2$ in $\Sigma$ contains exactly one positive edge, while the cycle $C_2$ in $-A_n$ is all-negative, $\Sigma$ is not switching equivalent to $-A_n$. 

    Note that this includes the case when $T$ is a path with $n \ge 4$.

    \begin{case}\label{case:1_nonleaf_endpoint}
        There is only one non-leaf of $T$ in $N(h)$, but it is an endpoint of some $\mathcal{P}$.
    \end{case}
    This case is similar to Case~\ref{case:2_nonleaves}. Note that this includes the case when $T$ is a path with $n = 3$.

    \begin{case}\label{case:1_nonleaf}
        There is only one non-leaf of $T$ in $N(h)$, which can never be an endpoint of a path in any possible $\mathcal{P}$.
    \end{case}
    Based on Case~\ref{case:2_nonleaves} and Case~\ref{case:1_nonleaf_endpoint}, $T$ has at least three leaves. Additionally, since the only non-leaf cannot be an endpoint,  all leaves are endpoints in any family of exactly $\lfloor \frac{n}{2} \rfloor$ edge- and endpoint-disjoint odd paths in $T$. Therefore, there are actually at least four leaves $a$, $b$, $c$ and $d$, all of which are endpoints in $\mathcal{P}$. Suppose that $a$ and $b$ are endpoints of the same odd path corresponding to the negative cycle set $\mathcal{C}$, and let the negative cycle containing $a$ and $b$ be $C_1$. Consider the signed graph $\Sigma = ( A_n, \sigma )$ where all edges are negative apart from $ha$ and $hb$. Then $\mathcal{C}$ is a family of exactly $\lfloor \frac{n}{2} \rfloor$ edge-disjoint negative cycles in $\Sigma$ (as the sign of the cycle $C_1$ remains negative), and hence $l( \Sigma ) = \lfloor \frac{n}{2} \rfloor$. Moreover, since the cycle containing $ha$ and $hc$ in $\Sigma$ contains exactly one positive edge, while this cycle in $-A_n$ is all-negative, $\Sigma$ is not switching equivalent to $-A_n$.

    \begin{case}\label{case:0_nonleaf}
        All vertices in $N(h)$ are leaves of $T$.
    \end{case}
    Based on Case~\ref{case:2_nonleaves} and Case~\ref{case:1_nonleaf_endpoint}, $T$ has at least three leaves. Similarly with Case~\ref{case:1_nonleaf}, picking two leaves $a$ and $b$, which are endpoints of the same odd path in $\mathcal{P}$ and negating the edges $ha$ and $hb$ leads to new signature maximises the frustration index, which is not switching equivalent to $-A_n$.
\end{proof}





\subsection{Fan Graphs} \label{section:fan}

We now focus on fan graphs, where $F_n = K_1 \vee P_n$, which is an apex tree.
In particular, we will be able to fully characterise and count the number of switching classes that maximise the frustration index of $F_n$.

From Theorem \ref{thm:apex_max_frustration}, we have 
$l_{\max}( F_n ) = \left\lfloor \frac{n}{2} \right\rfloor$. Moreover, since $P_n$ has a matching of size $\lfloor \frac{n}{2} \rfloor$, Theorem \ref{thm:apex_characterization} yields
\[
    l( -F_n ) =  l_{\max}( F_n ) = \left\lfloor \frac{n}{2} \right\rfloor.
\]
Thus, Theorem \ref{theorem:not_unique-max_apex_tree} shows that, unless $n \in \{1,2\}$, there is an antibalanced signature of $F_n$ that maximises the frustration index.

We already know that the frustration index of $(F_n, \sigma)$ is equal to the maximum number of disjoint negative cycles in $(F_n, \sigma)$. We now refine this result.


\begin{theorem}\label{thm:fan_disjoint_triangles}
    For the fan graph $F_n$, we have $l(F_n, \sigma) = p^-_{\triangle}(F_n, \sigma)$. 
\end{theorem}

\begin{proof}
    Let the only vertex in $K_1$ be $h$, and label the vertices in $P_n$ by $\{v_0, v_1, \cdots, v_{n-1}\}$. According to Theorem~\ref{thm:spanningtree}, we assume that $(F_n,\sigma)$ is switched so that all its negative edges are in $P_n$ without loss of generality.
    
    \begin{case}\label{case:fan_disjoint_1}
        There are no incident negative edges in $P_n$.
    \end{case}
    
    In this case, we have exactly $|E^-(F_n, \sigma)|$ edge-disjoint negative triangles, each with one negative edge. According to Equation~\ref{eq:bounds}, we have 
    \[
        l(F_n, \sigma) = |E^-(F_n, \sigma)| = p^-(F_n, \sigma) = p^-_{\triangle}(F_n, \sigma). 
    \]
    
    \begin{case}\label{case:fan_disjoint_2}
        There exist incident negative edges in $P_n$.
    \end{case}
    
    Let $(v_i, v_{i+1}, \cdots, v_{i+k-1})$ be a maximal all-negative path in $P_n$. 

    When $k$ is odd, switching the vertex subset $\{v_{i+1}, v_{i+3}, \cdots, v_{i+k-2}\}$ makes the path $(v_i, v_{i+1}, \cdots, v_{i+k-1})$ all-positive, and the edge subset $\{(v_{i+1}, h), (v_{i+3}, h), \cdots , (v_{i+k-2}, h)\}$ all-negative. Therefore, we get $\frac{k-1}{2}$ edge-disjoint triangles, each with one negative edge:
    \[
        (h, v_i, v_{i+1}, h), (h, v_{i+2}, v_{i+3}, h), \cdots, (h, v_{i+k-3}, v_{i+k-2}, h).
    \]

    When $k$ is even, switching $\{v_{i+1}, v_{i+3}, \cdots, v_{i+k-3}\}$ makes the edges in $(v_i, v_{i+1}, \cdots, v_{i+k-1})$ positive except the edge $(v_{i+k-2}, v_{i+k-1})$, and the edge subset $\{(v_{i+1}, h), (v_{i+3}, h), \cdots , (v_{i+k-3}, h)\}$ all-negative. Therefore, we get $\frac{k}{2}$ edge-disjoint triangles, each with one negative edge:
    \[
        (h, v_i, v_{i+1}, h), (h, v_{i+2}, v_{i+3}, h), \cdots, (h, v_{i+k-2}, v_{i+k-1}, h).
    \]

    After a series of such switching, we obtain a minimum $(F_n, \sigma)$ with exactly $|E^-(F_n, \sigma)| $ edge-disjoint negative triangles, each with one negative edge. Therefore, we have $l(F_n, \sigma) = |E^-(F_n, \sigma)| = p^-(F_n, \sigma) = p^-_{\triangle}(F_n, \sigma)$. 
\end{proof}

 \begin{cor}\label{cor:fan_negative_frustration}
     For a fan graph $F_n$, $l(F_n, \sigma) = \lfloor\frac{n}{2}\rfloor$ if and only if $(F_n, \sigma)$ contains exactly $\lfloor\frac{n}{2}\rfloor$ edge-disjoint negative triangles. 
 \end{cor}
 
We can further characterise the signatures maximising the frustration index combining Proposition~\ref{prop:-h-} and Corollary~\ref{cor:fan_negative_frustration}.

\begin{cor} \label{cor:Fn}
    For a fan graph $F_n$, $l(F_n, \sigma) = \lfloor\frac{n}{2}\rfloor$ if and only if $(F_n, \sigma)$  can be switched so that there are exactly $\lfloor\frac{n}{2}\rfloor$ edge-disjoint negative triangles, each with exactly and only one negative edge incident to $h$.
\end{cor}

Therefore, we can count the number of switching classes that maximise the frustration index.

\begin{cor}
For a fan graph $F_n$, there are $2^{\frac{n}{2}-1}$ switching classes maximising the frustration index when $n$ is even, and there are $\frac{n+3}{2}\times 2^{\frac{n-3}{2}}$ switching classes maximising the frustration index when $n$ is odd.
\end{cor}
\begin{proof}
    Let $\mathcal{F}_n$ be the set of $(F_n, \sigma)$ such that $(F_n, \sigma)$ is minimum and maximises the frustration index, and there are no negatively adjacent vertices in $V(P_n)$.

    Let $\Sigma = (F_n, \sigma) \in \mathcal{F}_n$. By Corollary \ref{cor:Fn}, there are exactly $\lfloor \frac{n}{2} \rfloor$ edge-disjoint negative triangles in $\Sigma$, and each contains exactly one negative edge (and that edge is adjacent to the hub $h$).
     
    Let us focus on the case where $n$ is even first.  Since there are $\frac{n}{2}$ edge-disjoint negative triangles in $\Sigma$, there must be exactly $\frac{n}{2} - 1$ edge-disjoint negative triangles in $\Sigma' = \Sigma - \{ v_0, v_1 \}$. Thus, $\Sigma'$ belongs to $\mathcal{F}_{n-2}$, and there are two choices for the extra negative edge between $v_0$ and $v_1$. We obtain
    \[
        |\mathcal{F}_n| = 2 \times |\mathcal{F}_{n-2}| \quad \text{if } n \text{ is even.}
    \]
    Since $|\mathcal{F}_2| = 2$. As a result, when $n$ is even, we have $|\mathcal{F}_n| = 2^{\frac{n}{2}}$. 

    Let us now consider the case where $n$ is odd. If $v_0$ or $v_1$ (but not both) is negatively adjacent to $h$, then $\Sigma' = \Sigma - \{ v_0, v_1 \}$ belongs to $\mathcal{F}_{n-2}$ as above. Otherwise, then $v_2$ must be negatively adjacent to $h$, for otherwise there wouldn't be $\lfloor \frac{n}{2} \rfloor$ edge-disjoint negative triangles. Then $(h, v_1, v_2)$ is a negative triangle, and $\Sigma'' = \Sigma - \{ v_0, v_1, v_2 \}$ belongs to $\mathcal{F}_{n-3}$. We obtain
    \[
        |\mathcal{F}_n| =  2 \times |\mathcal{F}_{n-2}| + |\mathcal{F}_{n-3}| \quad \text{if } n \text{ is odd}.
    \]
     Since $|\mathcal{F}_3| = 3$, when $n$ is odd, we have $|\mathcal{F}_n| = \frac{n+3}{2}\times 2^{\frac{n-3}{2}}$. 

    We also need to consider the switching equivalence. When $n$ is even, for each $(F_n, \sigma) \in \mathcal{F}_n$, we have $(F_n, \sigma^{\{h\}}) \in \mathcal{F}_n$, which means that $\mathcal{F}_n$ contains only $\frac{|\mathcal{F}_n|}{2}$ switching classes. On the contrary, when $n$ is odd, switching $\{h\}$ will increase the number of negative edges. In sum, we have $2^{\frac{n}{2}-1}$ switching classes maximising the frustration index when $n$ is even, and $\frac{n+3}{2}\times 2^{\frac{n-3}{2}}$ switching classes maximising the frustration index when $n$ is odd.
\end{proof}

\section{Wheel Graphs} \label{section:wheel_graphs}

In this section we discuss the frustration index of the wheel graph $W_n = K_1 \vee C_n$. 
We shall be able to derive results analogous to those for fan graphs.
The only outlier is the negative wheel graph $-W_n$ for $n$ odd.
This very special graph will allow to construct a counterexample to Conjecture \ref{conjecture:disjoint_cycles}, but for now let us consider the signed wheels.

Let the vertex in $K_1$ be $h$, and label the vertices in $C_n$ in clockwise way:
 \[
    V(C_n) = \{v_0, v_1, \cdots, v_i, \cdots, v_{n-1}\}, i\in\mathbb{Z}_n.
 \] 
 $V_{[i, j]}$ denotes the set of vertices $\{v_i, v_{i+1}, \cdots, v_j\}$.
\begin{lemma}\label{lemma:negative_wheel}
    The frustration index of $-W_n$ when $n$ is odd is
    \[
        l(-W_n) = \left\lceil{\frac{n}{2}}\right\rceil.
    \]
\end{lemma}

\begin{proof}
   We prove this theorem by directly finding a minimum signature $(W_n, \sigma)$ switching equivalent to $-W_n$ with $\left\lceil{\frac{n}{2}}\right\rceil$ negative edges.
   
   
   We first partition the vertices of $(W_n, \sigma)$ into three vertex subsets:
 \[
    \{h\} \cup N^+(h) \cup N^-(h), 
 \]
 where $N^+(h)$ is the set of all positive neighbours of $h$, and $N^-(h)$ is the set of all negative neighbours of $h$. Since $(W_n, \sigma)$ is switching equivalent to $-W_n$, all triangles in it are negative. Therefore, if two vertices in $N^-(h)$ are adjacent, they must be negatively adjacent. Similarly, if two vertices in $N^+(h)$ are adjacent, they must be negatively adjacent as well.
 

We now prove three claims:
 
 \begin{claim}
     Vertices in $N^-(h)$ cannot be adjacent.
 \end{claim}
 
 \begin{proof}
     Supposing that there are two adjacent vertices in $N^-(h)$, which can be $v_0$ and $v_1$ by symmetry, they must be negatively adjacent. Therefore, 
    \[
        d^-(v_0) = d^-(v_1) = 2 > d^+(v_0) = d^+(v_1) = 1,
    \]
    which is a contradiction.
 \end{proof}
 
 \begin{claim}\label{claim2}
     No more than two vertices in $N^+(h)$ can be consecutively adjacent.
 \end{claim}
 
 \begin{proof}
     Supposing that there are three consecutively adjacent vertices in $N^+(h)$, which can be $v_0$, $v_1$, and $v_2$ by symmetry, they must be negatively and consecutively adjacent. We have
    \[
        d^-(v_1)=2>d^+(v_1)=1
    \]
    which is a contradiction.
 \end{proof}
 
\begin{claim}
    There can be at most one pair of adjacent vertices in $N^+(h)$, which are negatively adjacent.
\end{claim}

\begin{proof}
    Supposing that $v_0$ and $v_1$ are adjacent vertices in in $N^+(h)$ by symmetry, and $v_i$ and $v_{i+1}$ is another pair of adjacent vertices in $N^+(h)$ such that $i$ is minimum.
    We have proved that vertices in $N^-(h)$ cannot be adjacent, so vertices in $N^+(h)$ and vertices $N^-(h)$ must appear alternatively in $V_{[2, i-1]}$. Since no more than two vertices in $N^+(h)$ can be consecutively adjacent, we have $v_2, v_{i-1} \in N^-(h)$. Thus,
     \[
    |N^-(h)\cap V_{[2, i-1]}|-|N^+(h)\cap V_{[2, i-1]}| = 1.
    \]
 
    Since $(v_0, v_1)$ and $(v_i, v_{i+1})$ are negative, we have
    \[
        d^-(V_{[1,i]}) - d^+(V_{[1,i]}) = 1,
    \]
    which is a contradiction.
\end{proof}
 
 Taken together, in $(W_n, \sigma)$, vertices in $N^+(h)$ and vertices in $N^-(h)$ alternate when we ignore the single possible pair of negatively adjacent vertices in $N^+(h)$. Since $n$ is odd and $d^-(h)\le\lfloor\frac{d(h)}{2}\rfloor=\frac{n-1}{2}$, there can be at most $\frac{n-1}{2}$ vertices in $N^-(h)$. Hence, there will be at least $\frac{n+1}{2}$ vertices in $N^+(h)$. Vertices in $N^+(h)$ and vertices in $N^-(h)$ can alternate only when there are two negatively adjacent vertices in $N^+(h)$ and $|N^-(h)| = \frac{n-1}{2}$.
 Therefore,
\[
    l(-W_n) = l(W_n, \sigma_0) = |E^-(W_n, \sigma_0)| = \lceil\frac{n}{2}\rceil.
\]
\end{proof}

Since every cycle must contain the vertex $h$ except $C_n$, we know that $p^-(-W_n) = p(W_n) = \lfloor \frac{n}{2} \rfloor$, which is the maximum number of edge-disjoint triangles in $W_n$. Therefore, when $n$ is odd, $l(-W_n) = p^-(-W_n) + 1$. It is easier to find the frustration index of $-W_n$ when $n$ is even.

\begin{lemma}\label{lemma:negative_wheel_even}
    The frustration index of $-W_n$ when $n$ is even is
    \[
        l(-W_n) = \frac{n}{2}.
    \]
\end{lemma}
\begin{proof}
    Switching the vertex subset $\{v_0, v_2, v_4, \cdots, v_{n-2}\}$ leads to a minimum signature $(W_n, \sigma_0)$ such that
    \[
        |E^-(W_n, \sigma_0)| = p^-(W_n, \sigma_0) = p^-_{\triangle}(W_n, \sigma_0) = \frac{n}{2}.
    \]
    Therefore, $ l(-W_n) = l(W_n, \sigma_0) = \frac{n}{2}$.
\end{proof}

\begin{theorem}\label{thm:wheel_disjoint_cycles}
    When $n$ is even, or when $n$ is odd and $(W_n, \sigma)$ is not switching equivalent to $-W_n$, $l(W_n, \sigma) = p^-(W_n, \sigma) = p^-_\triangle (W_n, \sigma)$. 
\end{theorem}
\begin{proof}
   We assume that every $(W_n, \sigma)$ is switched so that all negative edges are in $C_n$. There will be two cases.

    \begin{case}\label{case:wheel_disjoint_1}
        There are no incident negative edges in $C_n$.
    \end{case}
    
    In this case, we have exact $|E^-(W_n, \sigma)|$ edge-disjoint negative triangles, each with one negative edge. Therefore, 
    \[
        l(W_n, \sigma) = |E^-(W_n, \sigma)| = p^-(W_n, \sigma) = p^-_{\triangle}(W_n, \sigma).
    \]

    \begin{case}\label{case:wheel_disjoint_2}
        There exist incident negative edges in $C_n$.
    \end{case}

    If all edges in $C_n$ are negative, $(W_n, \sigma)$ is in the switching class with $-W_n$. 
    When $n$ is even, we have $l(W_n, \sigma)=\frac{n}{2} = p^-_{\triangle}(W_n, \sigma)$. When $n$ is odd, this is the only case that $l(W_n, \sigma) = p^-(W_n, \sigma) + 1$, which we have ruled out. 

    If there exists at least one positive edge in $C_n$, after a series of switching similar to Case~\ref{case:fan_disjoint_2} in Theorem~\ref{thm:fan_disjoint_triangles}, we obtain a minimum $(W_n, \sigma)$ with exactly $|E^-(W_n, \sigma)|$ edge-disjoint negative triangles, each with one negative edge.
\end{proof}

For any signature $(W_n, \sigma)$, it is trivial that $p^-(W_n, \sigma) = p^-_{\triangle}(W_n, \sigma) \le p(W_n) = \lfloor \frac{n}{2} \rfloor$. Therefore, we have the following corollary. 

\begin{cor}\label{cor:wheel_max_frutration}
    For the wheel graph, we have $l_{\max}(W_n) = l(-W_n) = \lceil \frac{n}{2} \rceil.$ Additionally, when $n$ is odd, $-W_n$ is in the unique switching class that maximises the frustration index. When $n$ is even, $(W_n, \sigma)$ maximises the frustration index if and only if it contains $\frac{n}{2}$ edge-disjoint negative triangles.
\end{cor}

    

We can then count the switching classes that maximise the frustration index.

\begin{cor}
    For the wheel graph $W_n$ when $n$ is even, there are $2^{\frac{n}{2}+1}-1$ switching classes that maximise the frustration index.
\end{cor}
\begin{proof}
    All edges that are adjacent to $h$ form a spanning tree of $W_n$. Therefore, the set of all triangles is a fundamental cycle set. We have only two ways to allocate $\frac{n}{2}$ edge-disjoint triangles:
    \[
        \{(h, v_0, v_1, h), (h, v_2, v_3, h), \cdots, (h, v_{n-2}. v_{n-1}, h) \},
    \]
    and
    \[
        \{(h, v_1, v_2, h), (h, v_3, v_4, h), \cdots, (h, v_{n-1}, v_0, h)\}.
    \]
    For each allocation, we can construct $2^{\frac{n}{2}}$ switching classes that maximise the frustration index by assigning different signs (positive or negative) to the remaining $\frac{n}{2}$ triangles. They are different switching classes for they do not share the same set of negative fundamental cycles. However, for the $2^{\frac{n}{2}}$ switching classes corresponding to the first allocation and the $2^{\frac{n}{2}}$ switching classes corresponding to the second allocation, there is one repetition, which is $-W_n$, where all triangles are negative. Therefore, there are $2^{\frac{n}{2}+1}-1$ switching classes that maximise the frustration index.
\end{proof}

We can now construct a counterexample to Conjecture~\ref{conjecture:disjoint_cycles}; it is the signed graph $\Sigma$ in Figure~\ref{fig:disjoint_cycles}. We note that $|\Sigma|$ is indeed $2$-connected, and has minimum degree $3$, hence it satisfies our assumption. 

\begin{figure}[htbp]    
    \centering
    \includegraphics[width=0.3\linewidth]
    {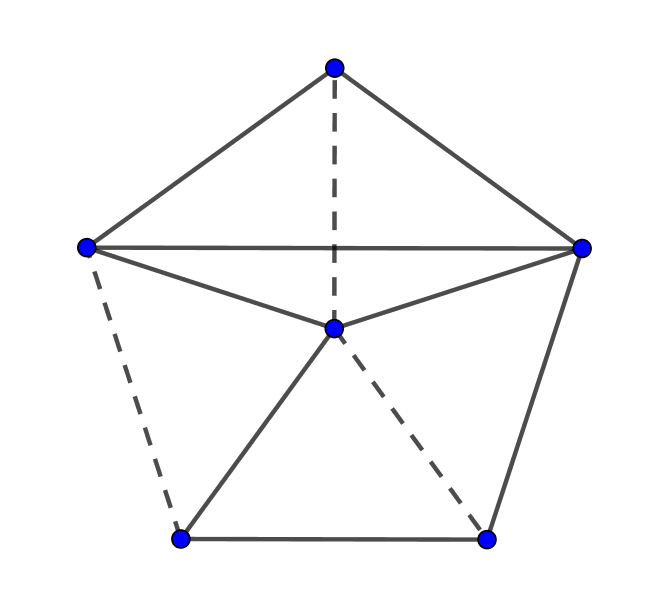}
    \caption{Counterexample $\Sigma$ to Conjecture \ref{conjecture:disjoint_cycles}}
    \label{fig:disjoint_cycles}
    \end{figure}

The signed graph $\Sigma$ is formed by connecting two non-adjacent vertices in a minimum signed wheel graph $(W_5, \sigma)$ in the switching class of $-W_5$. Since  $l(-W_5) = 3$, we have $l(\Sigma) \ge 3$ (and since $\Sigma$ only has three negative edges, we obtain $l( \Sigma ) = 3$). Additionally, it is not hard to check that $p^-(\Sigma) = 2$, while $p_\triangle( \Sigma ) = p(\Sigma) = 3$. Therefore, $p^-(\Sigma) < \min (l(\Sigma), p_\triangle(\Sigma))$. In fact, $\Sigma$ is planar and thus this conjecture does not even hold for planar graphs.

\section{Complete Split Graphs} \label{section:complete_split_graphs}

In this section, we consider complete split graphs $S_n^m = K_n \vee \overline{K}_m$. 
We have already seen that $S_2^3 = K_1 \vee S_4$ is a graph for which the negative signature does not maximise the frustration index.
However, we will prove that for $n$ large enough, both $-S_n^2$ and $-S_n^3$ maximise the frustration index uniquely.

We begin by determining the negative frustration index of any complete split graph.

\begin{lemma}\label{lemma:negative_split}
    The frustration index of the all-negative complete split graph $-S_n^m=-(K_n\vee\overline{K}_m)$ is 
    \begin{equation}
    l(-S_n^m) = \left\{
    \begin{array}{lr}
        \lfloor\frac{(n-1)^2}{4}+\frac{mn}{2}-\frac{m^2}{4}\rfloor, n > m, \\
        \frac{n(n-1)}{2}, n \le m.
    \end{array}
    \right.
    \end{equation}
\end{lemma}

\begin{proof}
    Since $S_n^m$ is highly symmetric, we first partition $V(K_n)$ into two sets $V_i$ and $V_{n-i}$, one with $i$ vertices and the other one with $(n-i)$ vertices where $\lceil \frac{n}{2} \rceil \le i \le n$. Then, we partition $V(S_n^m)$ into two vertex subsets:
    \[
        V(S_n^m) = V_{i}\cup (V_{n - i}\cup V(\overline{K_m}))
    \]
    The maximum number of edges between the above two vertex subsets is the number of edges of the max-cut of $S_n^m$: 
    \begin{align*}
        \max_{X\subseteq V(S_n^m)}|E(X, \overline{X})| & = \max_{\lceil \frac{n}{2} \rceil\le i \le n}\{i \times(n-i)+i\times m\}\\
        & =\max_{\lceil \frac{n}{2} \rceil\le i \le n} \{ -(i - \frac{m+n}{2})^2 + \frac{(m+n)^2}{4}\}.
    \end{align*}
    Therefore, 
    \begin{align*}
        \max\limits_{X\subseteq V(S_n^m)}|E(X, \overline{X})| = \left\{
    \begin{array}{lr}
        \lceil\frac{(m+n)^2-1}{4}\rceil, n > m, \\
        mn, n \le m.
    \end{array}
    \right.
    \end{align*}
    
    Now we get the frustration index of $-S_n^m$ :
    \begin{equation}
    l(-S_n^m) = E(S_n^m)-\max_{X\subseteq V(S_n^m)}|E(X, \overline{X})| = \left\{
    \begin{array}{lr}
        \lfloor\frac{(n-1)^2}{4}+\frac{mn}{2}-\frac{m^2}{4}\rfloor, n > m, \\
        \frac{n(n-1)}{2}, n \le m.
    \end{array}
    \right.
    \end{equation}
\end{proof}

The following remark is derived from the above proof, which gives us a characterisation of the max-cuts of $S_n^m$ and thus the minimum signatures of $-S_n^m$.
\begin{remark}\label{remark:minimum_negative_split}
    When $n>m$ and $(m+n)$ is even, there is one single way to gain a max-cut: 
\[
    V(S_n^m) = V_{\frac{n+m}{2}}\cup (V_{\frac{n-m}{2}}\cup V(\overline{K_m})),
\]
where $i = \frac{m+n}{2}$.

When $n>m$ and $(m+n)$ is odd, there are two ways to gain a max-cut: 
\[
    V(S_n^m) = V_{\frac{n+m+1}{2}}\cup (V_{\frac{n-m-1}{2}}\cup V(\overline{K_m})),
\]
where $i = \frac{m + n + 1}{2}$, or
\[
    V(S_n^m) = V_{\frac{n+m-1}{2}}\cup (V_{\frac{n-m+1}{2}}\cup V(\overline{K_m})),
\]
where $i = \frac{m + n - 1}{2}$.

When $n \le m$, there is one single way to gain a max-cut:
\[
    V(S_n^m) = V(K_n)\cup V(\overline{K_m}),
\]
where $i = n$.
\end{remark}

We focus on the case when $n > m$. The following Proposition show that we only need to exhibit one $n_0 > m$ such that $-S_{n_0}^m$ maximises the frustration index uniquely to draw the same conclusion for all $n \ge n_0$.

\begin{prop}\label{prop:S_n+k+1^n}
    Let $k \ge 1$. For any $m \ge 2$, if $-S_{m+k}^m$ maximises the frustration index of $S_{m+k}^m$ uniquely, then $-S_{m+k+1}^m$ maximises the frustration index of $S_{m+k+1}^m$ uniquely.
    
\end{prop}
\begin{proof}
    It is not hard to check that 
    \begin{align*}
        l(-S_{m+k+1}^m) &= \left\lfloor \frac{(m+k)^2}{4} + \frac{m(m+k+1)}{2} - \frac{m^2}{4} \right\rfloor \\
        &= \left\lfloor \frac{(m+k-1)^2}{4} + \frac{m(m+k)}{2} - \frac{m^2}{4} \right\rfloor  + \left\lfloor \frac{2m+k}{2} \right\rfloor\\
        &= l(-S_{m+k}^m) + \left\lfloor \frac{2m+k}{2} \right\rfloor.
    \end{align*}
    $S_{m+k+1}^m$ is formed by adding a dominating vertex to $S_{m+k}^m$ and $|E(S_{m+k}^m)| = 2m+k$. Therefore, according to Theorem~\ref{thm:disjoint_vertex_subset}, if $l_{\max}(S_{m+k}^m) = l(-S_{m+k}^m) = \left\lfloor \frac{(m+k-1)^2}{4} + \frac{m(m+k)}{2} - \frac{m^2}{4} \right\rfloor $, then $l_{\max}(S_{m+k+1}^m) = l(-S_{m+k+1}^m) = \left\lfloor \frac{(m+k)^2}{4} + \frac{m(m+k+1)}{2} - \frac{m^2}{4} \right\rfloor$.

    Now we prove the uniqueness. For any signature $(S_{m+k+1}^m, \sigma')$ which is not switching equivalent to $-S_{m+k+1}^m$, there is at least one positive triangle in $(S_{m+k+1}^m, \sigma')$. Let the triangle be $(a, b, c, a)$. Removing a vertex $v \in V(K_{m+k+1}) \setminus \{a, b, c\}$ leads to a signed subgraph $(S_{m+k}^m, \sigma')$ which contains the positive triangle. Therefore, $(S_{m+k}^m, \sigma')$ is not switching equivalent to $-S_{m+k}^m$ and $l(S_{m+k}^m, \sigma') < l_{\max}(S_{m+k}^m)$. As a result, 
    \[
        l(S_{m+k+1}^m, \sigma') \le l(S_{m+k}^m, \sigma') + \left\lfloor \frac{2m+k}{2} \right\rfloor <  l_{\max}(S_{m+k}^m) + \left\lfloor \frac{2m+k}{2} \right\rfloor = l_{\max}(S_{m+k+1}^m),
    \]
    which means that $(S_{m+k+1}^m, \sigma')$ does not maximise the frustration index.
\end{proof}


We first settle the case for $m = 2$, with $n_0 = 3$.

\begin{lemma}\label{lemma:s_3^2}
    For the complete split graph $S_3^2 = K_3 \vee \overline{K_2}$, the all-negative signature maximises the frustration index uniquely, and
    \[
         l_{\max} (S_3^2) = l(-S_3^2) = 3.
    \]
\end{lemma}
\begin{proof}
    Since $S_3^2$ can be seen as a wheel graph $W_4$ with two non-adjacent vertices being connected, we know that $l_{\max}(S_3^2) \le l_{\max}(W_4) + 1 = 3$. Therefore, $l_{\max} (S_3^2) = l(-S_3^2) = 3$. Additionally, when $\Sigma = (S_3^2, \sigma)$ contains a positive triangle with vertex set $T$ and $S = V(\Sigma) \setminus T$, $\Sigma[S]$ is a path or $\overline{K_2}$. When $\Sigma[S]$ is a path, which means $|T \cap V(K_3)| = 2$, we have $d(T) = 5$. Therefore, $l(\Sigma) \le 2$ according to Theorem~\ref{thm:disjoint_vertex_subset}. When $\Sigma[S]$ is $\overline{K_2}$, $\Sigma$ can be switched so that $\Sigma[T]$ is an all-positive triangle, and $d^-(T) \le \frac{d(T)}{2} = 3$. However, since every vertex in $V(\overline{K_2}) = S$ has degree only $3$, $\Sigma$ is not minimum if $d^-(S) = d^-(T) = 3$. Therefore, $l(\Sigma) \le 2$.
\end{proof}

Combining Proposition~\ref{prop:S_n+k+1^n} and Lemma~\ref{lemma:s_3^2}, we have the following theorem.

\begin{theorem}\label{thm:S_n^2}
    For the complete split graph $S_n^2$ where $n \ge 3$, the all-negative signature maximises the frustration index uniquely, and 
    \[
        l_{\max}(S_n^2) = l(-S_n^2) = \left\lfloor \frac{(n-1)^2}{4} + n - 1\right\rfloor.
    \]
\end{theorem}

We now settle the case $m = 3$, with $n_0 = 5$.




\begin{lemma}\label{lemma:S_5^3}
    For the complete split graph $S_5^3 = K_5 \vee \overline{K_3}$, the all-negative signature maximises the frustration index uniquely.
\end{lemma}
\begin{proof}
    Let $(S_5^3, \sigma)$ be a minimum signature. Note that all vertices in $S_5^3$ have an odd degree. Let $X \subseteq V(S_5^3, \sigma)$ be the set of vertices such that $d^-(v) = \frac{d(v)-1}{2}$ for every vertex $v \in X$, and $Y \subseteq V(S_5^3, \sigma)$ the set of vertices such that $d^-(v) \le \frac{d(v)-3}{2}$ for every vertex $v \in Y$.

    We have the following claim.

    \begin{claim}\label{claim:positive_p3}
        There cannot exist an all-positive path $(a, b, c)$ whose vertices are in $X$ and two endpoints $a, c\in V(\overline{K_3})$, otherwise $d^-(\{a, b, c\}) - d^+(\{a, b, c\}) = 1$. 
    \end{claim}
    \begin{proof}
        Supposing that such path exists, we have $d^+(v) - d^-(v) = 1$ for every $v \in \{a, b, c\}$. Therefore,
        \[
            d^-(\{a, b, c\}) - d^+(\{a, b, c\}) = (d^-(a) + d^-(b) + d^-(c)) - (d^+(a) + d^+(b) + d^+(c) - 4) = 1,
        \]
        which is a contradiction.
    \end{proof}

    If $X = V(S_5^3, \sigma)$,
     \[
         |E^-(S, \sigma)| = \frac{5 \times 3 + 2 \times 3}{2}=\frac{21}{2},
     \]
    which is a contradiction. Therefore, $|Y| \ge 1$, and $l_{\max}(S_5^3) \le 10$. It is not hard to check that, for a minimum signature $(S_5^3, \sigma_0)$ that is switching equivalent to $-S_5^3$, $|E^-(S_5^3, \sigma_0)| = l(-S_5^3) = 9$. Now we prove this theorem by proving that whenever there is a positive triangle in $(S_5^3, \sigma)$, $l(S_5^3, \sigma) \le 8$.

    Let $(S_5^3, \sigma)$ be a minimum signature with a positive triangle $(a, b, c, a)$, $T = \{a, b, c\}$, and $S = V(S_5^3, \sigma) \setminus T = \{v_1, v_2, v_3, v_4, v_5\}$. Since $S$ and $T$ are disjoint, we can assume that $\Sigma = (S_5^3, \sigma)$ is switched so that $\Sigma[T]$ is all-positive, $\Sigma[S]$ is minimum (any minimum signature in the switching class), and $d^-(T) < d^+(T)$.

    There are two cases for $T$.
    



    \begin{case}\label{case:T_notin_K_5}
        $T \cap V(V_5) = \{a, b\}$, and $T \cap V(\overline{K_3}) = \{c\}$. 
    \end{case}
    In this case, $|\Sigma[S]| = S_3^2 = K_3 \vee \overline{K_2}$. According to Theorem~\ref{thm:disjoint_vertex_subset}, $l(\Sigma) \le 9$. 

    Suppose that $l(\Sigma) = 9$, and $\Sigma$ is minimum. According to Lemma~\ref{lemma:s_3^2} and Theorem~\ref{thm:disjoint_vertex_subset}, $\Sigma[S]$ must be switching equivalent to $-S_3^2$, and $d^-(T) = 6$. In fact, $\Sigma[S]$ can be any minimum signature in the switching class of $-S_3^2$.
    
    Let $S \cap V(K_5) = \{v_1, v_2, v_3\}$, and $S \cap V(\overline{K_3}) = \{v_4, v_5\}$. $\Sigma[S]$ can be the signature such that the triangle $(v_1, v_2, v_3, v_1)$ is all-negative, and all the edges between $\{v_1, v_2, v_3\}$ and $\{v_4, v_5\}$ are positive according to Remark~\ref{remark:minimum_negative_split}. 
    
    Supposing that $|S \cap Y| \ge 2$, we have 
    \[
        \sum_{i=1}^{5}d^-(v_i) \le \sum_{v_i\in S \cap X} \frac{d(v_i)-1}{2} + \sum_{v_i\in S \cap Y} \frac{d(v_i)-3}{2} =  \sum_{i=1}^{5}\frac{d(v_i)}{2} - \frac{9}{2} = 11.
    \]
    Since the triangle $(v_1, v_2, v_3, v_1)$ is all-negative in $\Sigma[S]$, we have $d^-(T) = d^-(S) =  \sum_{i=1}^{5}d^-(v_i) - 6 = 5$, which is a contradiction. Supposing that $S \subseteq X$, we have $\sum_{i=1}^{5}d^-(v_i) = \sum_{i=1}^{5}\frac{d(v_i)}{2} - \frac{5}{2} = 13$ and thus $d^-(T) = 7$, which is a contradiction. Therefore, $|S \cap Y| = 1$. Additionally, supposing that the only vertex $v \in S \cap Y$ satisfies $d^-(v) \le \frac{d(v)-5}{2}$, we also have $\sum_{i=1}^{5}d^-(v_i) \le \sum_{i=1}^{5}\frac{d(v_i)}{2} - \frac{9}{2} = 11$, which is a contradiction. Taken together, We have $|S \cap Y| = 1$ and the only vertex $v \in S \cap Y$ satisfies $d^-(v) = \frac{d(v)-3}{2}$. 
    
    The only vertex in $S \cap Y$ must be in $S \cap V(\overline{K_3})$ according to Claim~\ref{claim:positive_p3}. Let $v_5 \in S \cap Y$ by symmetry, and then we have $\{v_1, v_2, v_3, v_4\} \subseteq X$. Therefore, $v_4$ must be negatively adjacent to both $a$ and $b$, and $v_5$ is negatively adjacent to exactly one vertex in $\{a, b\}$, which can be $b$ by symmetry. Since $d^-(c) \le 2$, at least one vertex in $\{v_1, v_2, v_3\}$ is positively adjacent to $c$, say $v_1$ by symmetry. Since $v_1$ is positively adjacent to both $c$ and $v_4$, we have $c \in Y$ according to Claim~\ref{claim:positive_p3}, which means $d^-(c) \le 1$. 
    Now we discuss by two cases, as shown in Figure~\ref{fig:s53_uniqueness1} and Figure~\ref{fig:s53_uniqueness2}. Note that they share the same black edges.

    \begin{subcase}\label{subcase:0}
        $d^-(c) = 0$.
    \end{subcase}

    In this case, we have $a, b \in X$, otherwise $d^-(T) \le 5$.     Note that each vertex $ v \in \{v_1, v_2, v_3\}$ is negatively adjacent to exactly one vertex in $T$ as $v \in X$ and  $(v_1, v_2, v_3, v_1)$ is all-negative. Since $b$ is negatively adjacent to both $v_4$ and $v_5$, $b$ must be negatively adjacent to exactly one vertex in $\{v_1, v_2, v_3\}$, which can be $v_3$ by symmetry as now $c$ is positively adjacent to all vertices in $\{v_1, v_2, v_3\}$. Additionally, since $a$ is negatively adjacent to $v_4$ and positively adjacent to $v_5$, $a$ must be negatively adjacent to two vertices in $\{v_1, v_2, v_3\}$, which can only be $v_1$ and $v_2$. Therefore, there is only one possible minimum signature with respect to isomorphism, as shown in Figure~\ref{fig:s53_uniqueness1}. In this case, $d^-(\{a, v_3, v_5\}) = 7 > d^+(\{a, v_3, v_5\}) = 6$, which is a contradiction.

    \begin{subcase}\label{subcase:1}
        $d^-(c) = 1$.
    \end{subcase}
    
    In this case, $c$ is negatively adjacent to one vertex in $\{v_2, v_3\}$, which can be $v_2$ by symmetry. In this case, $d^-(\{a, b\}) = 5$, and $v_2$ is positively adjacent to both $a$ and $b$. Therefore, $d^-(\{a, b, v_2\}) = d^-(\{a, b\}) + d^-(v_2) = 8 > d^+(\{a, b, v_2\}) = 7$, which is a contradiction.

    Taken together, $l(\Sigma) \le 8$.
    
    \begin{figure}[htbp]
        \centering
        \begin{subfigure}{0.4\textwidth}
            \centering\includegraphics[width=\textwidth]{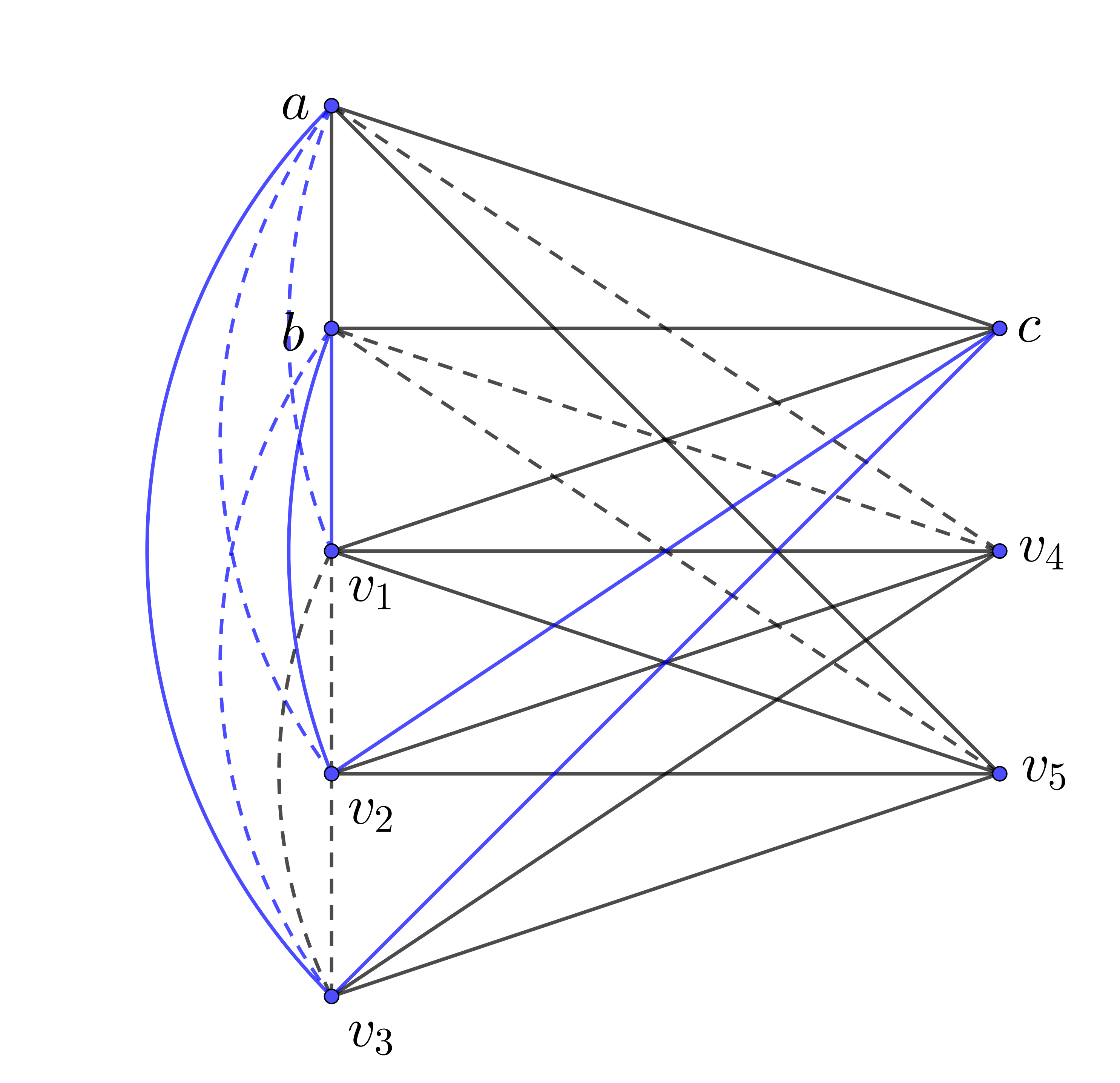}
            \caption{$d^-(c) = 0$}
            \label{fig:s53_uniqueness1}
        \end{subfigure}
        \begin{subfigure}{0.4\textwidth}
            \centering\includegraphics[width=\textwidth]{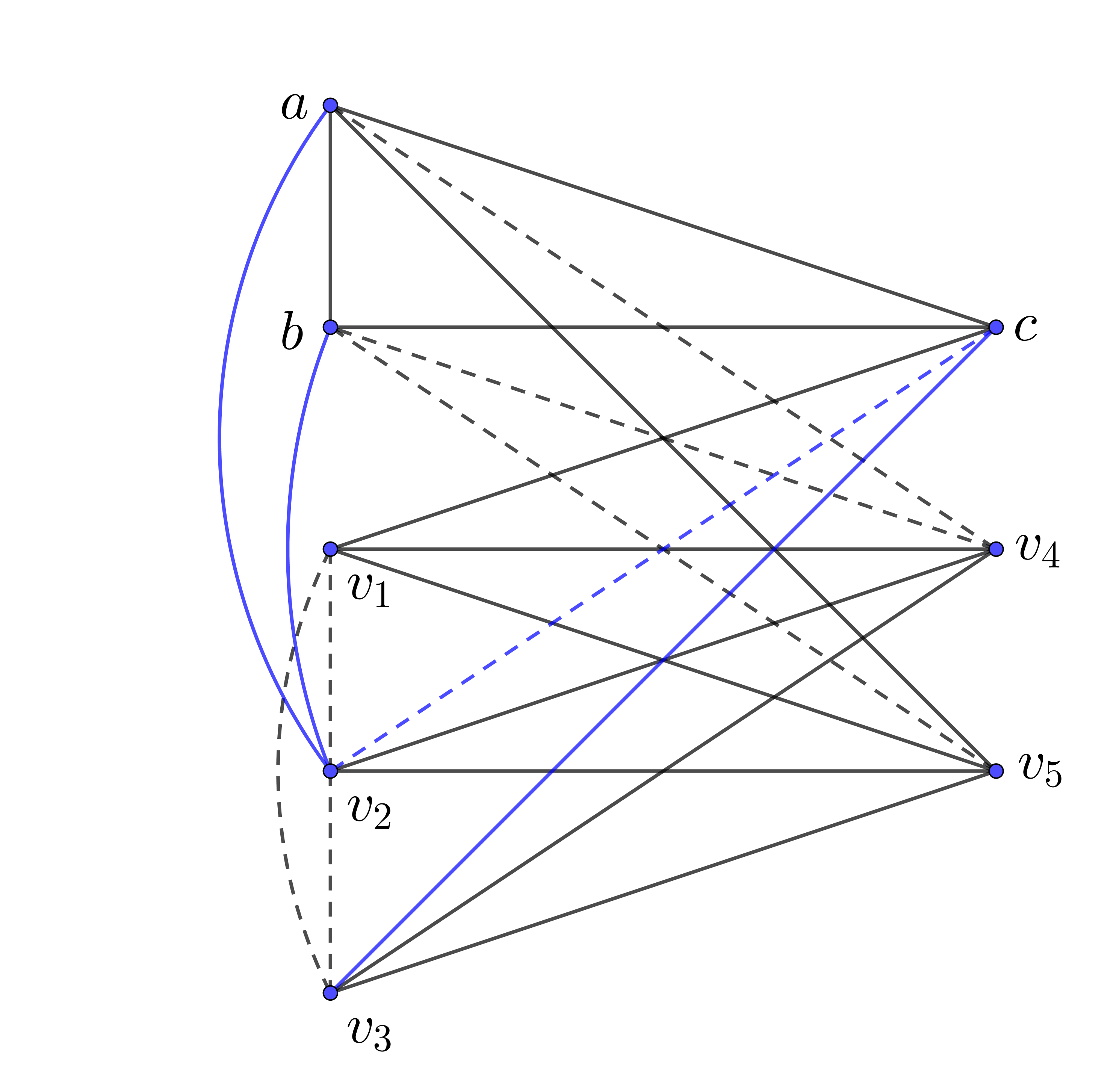}
            \caption{Part of the edges when $d^-(c) = 1$}
            \label{fig:s53_uniqueness2}
        \end{subfigure}
        \caption{Two subcases of Case~\ref{case:T_notin_K_5}}
        \end{figure}        

    \begin{case}\label{case:T_in_K_5}
        $T \subseteq V(K_5)$.
    \end{case}
    In this case, $|\Sigma[S]| = K_1 \vee S_3$. According to Corollary~\ref{corollary:star} and Theorem~\ref{thm:disjoint_vertex_subset}, $l(\Sigma) \le 9$.

    Suppose that $l(\Sigma) = 9$, and $\Sigma$ is minimum. According to the discussion in Case~\ref{case:T_notin_K_5}, there cannot be a positive triangle with two vertices in $V(K_5)$ and one vertex in $V(\overline{K_3})$. For a vertex $v \in V(\overline{K_3})$, $v$ must be negatively adjacent to exactly one vertex in $\{a, b\}$, exactly one vertex in $\{b, c\}$, and exactly one vertex in $\{a, c\}$, which is impossible. Therefore, we have a contradiction, and $l(\Sigma) \le 8$.

\end{proof}

\begin{theorem}\label{thm:S_n^3}
    For the complete split graph $S_n^3$ where $n \ge 5$, the all-negative signature maximises the frustration index uniquely, and
    \[
    l_{\max}(S_n^3) = l(-S_n^3) = \left\lfloor \frac{(n-1)^2}{4} + \frac{3n}{2} - \frac{9}{4}\right\rfloor.
\]
\end{theorem}

\begin{remark}\label{remark:S43}
    When $\Sigma =(S_4^3, \sigma)$ contains a positive triangle with vertex set $T$ and $S = V(\Sigma) \setminus T$, $\Sigma[S]$ is either isomorphic to the fan graph $F_3$ or isomorphic to the star $S_3$. According to Theorem~\ref{thm:disjoint_vertex_subset}, we have $l(\Sigma) \le 6$. Therefore, $l_{\max}(S_4^3) = l(-S_4^3) = 6$. However, the all-negative does not maximise the frustration index uniquely. Figure~\ref{fig:counter_exampleS4} shows a minimum signature maximising the frustration index. This signature is not switching equivalent to the all-negative signature, for it contains a positive triangle.
\end{remark}
\begin{figure}[htbp]    
    \centering
    \includegraphics[width=0.33\linewidth]
    {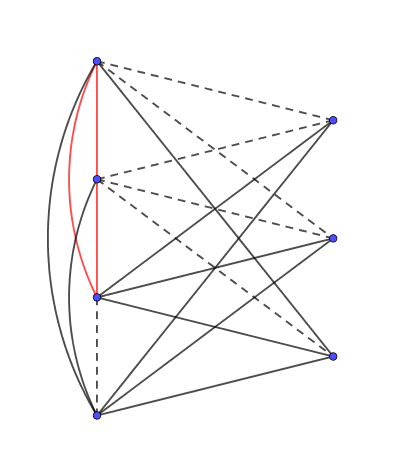}
    \caption{Another signature maximising the frustration index of $S_4^3$}
    \label{fig:counter_exampleS4}
    \end{figure}

As an open problem for future work, we ask: for all $m$, does there exist $n_0 > m$ such that $-S_{n_0}^m$ maximises the frustration index uniquely?

\end{document}